\documentclass{amsart}

\usepackage{geometry}
\usepackage{amsmath,bm}
\usepackage{mathrsfs}
\usepackage{amsfonts}
\usepackage{indentfirst}
\usepackage{leftidx}
\usepackage{amssymb}
\usepackage{amsthm}
\usepackage[all]{xy}

\usepackage{verbatim}
\usepackage{cite}
\usepackage[latin1]{inputenc}
\usepackage{color}
\usepackage{float}
\usepackage{subfigure}
\allowdisplaybreaks[4]

\def\mono{\bm{i},\bm{a}}
\def\K{\text{\rm{Ker }}}

\newcommand{\bmk}{{\bm{k}}}

\newtheorem{Thm}{Theorem}[section]
\newtheorem{Rk}[Thm]{Remark}
\newtheorem{Def}[Thm]{Definition}
\newtheorem{Prop}[Thm]{Proposition}
\newtheorem{Lem}[Thm]{Lemma}
\newtheorem{Cor}[Thm]{Corollary}

\begin{document}

\title{Higher Cohomology Vanishing of line bundles
on generalized Springer resolution}

\author{Yue Hu}
\address{Academy of Mathematics and Systems Science, University of Chinese Academy of Sciences, Beijing, China}
\email{huyue@amss.ac.cn}
\date{}

\begin{abstract}
We give a proof of a conjecture raised by Michael Finkelberg and Andrei Ionov\cite{FI} based on a certain generalization of Springer resolution and Grauert-Riemenschneider vanishing theorem. As a corollary, the coefficients of multivariable version of Kostka functions introduced by Finkelberg and Ionov are non-negative.
\end{abstract}

\maketitle

\section{Introduction}
In \cite{FI}, Finkelberg and Ionov proposed the following conjecture about higher cohomology vanishing of line bundles.
Let $r\geqslant 2$ be an integer, $\bm{k}$ an algebraically closed field.
Let $A_{1\leqslant i\leqslant r-1}$ be a collection of upper triangular $N\times N$ matrices over $\bm{k}$, let $A_{r}$ be a strictly upper triangular $N\times N$ matrix over $\bm{k}$, and let
$B_{N}$ be the group of upper triangular matrices in $GL_{N}$.
The group $B_{N}^{r}$ acts on the vector space formed by the $r$-tuples
$\bm{\mathfrak{n}_{r}}:=\{(A_{1},\dots,A_{r})\}$ by
\[(b_{1},\dots,b_{r})\cdot(a_{1},\dots,a_{r})=(b_{1}a_{1}b_{2}^{-1},\dots,b_{r}a_{r}b_{1}^{-1})\]
Consider the associated vector bundle $\mathcal{T}_{r}^{*}\mathcal{B}_{N}^{r}:=GL_{N}^{r}\times^{B_{N}^{r}}\bm{\mathfrak{n}_{r}}$ over the flag variety $\mathcal{B}_{N}^{r}$.
Let $\mathcal{O}(\bm{\lambda})$ be the pull back of a line bundle $\mathcal{L}(\bm{k}_{\bm{-\lambda}})$ over $\mathcal{B}_{N}^{r}$ to $\mathcal{T}_{r}^{*}\mathcal{B}_{N}^{r}$, where $\bm{\lambda}$ is in the weight lattice of $GL_{N}^{r}$.
The conjecture claims that $H^{i>0}(\mathcal{T}_{r}^{*}\mathcal{B}_{N}^{r},\mathcal{O}(\bm{\lambda}))=0$ for any dominant weight $\bm{\lambda}$.
In this paper, we prove this vanishing theorem both for the characteristic $0$ and $p>0$ case.

Finkelberg and Ionov proposed this vanishing theorem for the goal of generalizing Kostka polynomials to multivariable case.
Let us give some historic remarks and fix some notations.
Let $G$ be a connected split reductive group over $\bm{k}$.
Choose a pair $T\subset B$ of a maximal torus and a Borel subgroup corresponding to the positive roots. We consider the flag variety $G/B$.
For any weight $\lambda\in X(T)$, let $\bm{k}_{\lambda}$ be the corresponding one dimensional $T$-module: $t\cdot\upsilon_{\lambda}=\lambda(t)\cdot\upsilon_{\lambda}$.
Denote by $\mathcal{L}_{G/B}(\bm{k}_{\lambda})$ the associated sheaf of the homogeneous line bundle $G\times^{B}\bm{k}_{\lambda}$ on $G/B$.
We have $\mathcal{L}_{G/B}(\bm{k}_{-\lambda})\simeq\mathcal{L}_{G/B}(\bm{k}_{\lambda})^{\vee}$ canonically, where $\mathcal{L}_{G/B}(\bm{k}_{\lambda})^{\vee}$ is the dual line bundle.
The Borel-Bott-Weil theorem describes the higher cohomology of the line bundle $\mathcal{L}_{G/B}(\bm{k}_{-\lambda})$.
In particular, the Kempf vanishing theorem says $H^{>0}(G/B,\mathcal{L}_{G/B}(\bm{k}_{-\lambda}))=0$ if $\lambda$ is dominant.

The cotangent bundle $\pi: T^{*}G/B:=G\times^{B}\mathfrak{n}\rightarrow G/B$ of the flag variety has many favorable properties.
By Springer's theory, $\varphi: T^{*}G/B\rightarrow \mathcal{N}$ is the resolution of singularities of the nilpotent cone $\mathcal{N}$, where $\varphi$ is the moment map
\[\psi: G\times^{B}\mathfrak{n}\rightarrow \mathcal{N}:\ (g,x)\mapsto gxg^{-1}\]
It is interesting that the higher cohomology of the pull back of dominant line bundles on $G/B$ to $T^{*}G/B$ also vanishes, namely, we have  $H^{>0}(T^{*}G/B,\mathcal{L}_{T^{*}G/B}(\bm{k}_{-\lambda}))=0$ if $\lambda$ is dominant (we denote by  $\mathcal{L}_{T^{*}G/B}(\bm{k}_{-\lambda}):=\pi^{*}\mathcal{L}_{G/B}(\bm{k}_{-\lambda})$).
When $\bm{k}$ is of characteristic zero, this vanishing theorem was proved by Broer\cite{broer1}, \cite{broer2} using two different methods.
When $\bm{k}$ is of characteristic $p>0$, this theorem was proved by Andersen and Jantzen \cite{AJ} for classical groups and later by Kumar, Lauritzen, and Thomsen \cite{KLT} for arbitrary semisimple and simply connected algebraic groups.

When $\bm{k}$ is of characteristic zero, using the above vanishing theorem, Brylinski \cite{bry} gave a description of the coefficients of Lusztig's \cite{L81} $q$-analog $m_{\lambda}^{\mu}(q)$ of dominant $\mu$-weight multiplicity in a finite dimensional irreducible $\mathfrak{g}$-representation $V_{\lambda}$.
The cohomology groups of line bundles on the cotangent bundle of flag variety have a natural structure of a graded $G$-module given by
\[H^{i}(T^{*}G/B,\pi^{*}\mathcal{L}_{G/B}(\bm{k}_{\lambda}))=\bigoplus_{j=0}^{\infty}H^{i}(G/B,\mathcal{L}_{G/B}(S^{j}\mathfrak{n}^{*}\otimes\bm{k}_{\lambda}))\]
Hence the graded Euler characteristic character associated to $\pi^{*}\mathcal{L}_{G/B}(\bm{k}_{-\mu})$ equals to
\[\chi_{T^{*}G/B}^{-\mu}:=\chi(T^{*}G/B,\pi^{*}\mathcal{L}_{G/B}(\bm{k}_{-\mu}))=\sum_{i,j\geqslant 0}(-1)^{i}\text{ch}(H^{i}(G/B,\mathcal{L}_{G/B}(S^{j}\mathfrak{n}^{*}\otimes \bm{k}_{-\mu})))q^{j} \]
Brylinski showed that for all $\mu\in X(T)$ we have $\chi_{T^{*}G/B}^{-\mu}=\sum_{\lambda\in X^{+}(T)}m_{\lambda}^{\mu}(q)\chi_{\lambda}^{*}$, where $\chi_{\lambda}^{*}$ is the character of the dual of the irreducible $G$-representation $V_{\lambda}$.
Hence the higher cohomology vanishing of $\mathcal{L}_{T^{*}G/B}(\bm{k}_{-\mu})$ for any dominant $\mu$ says that
\[\sum_{j\geqslant 0}\langle V_{\lambda}^{*},\Gamma(G/B,\mathcal{L}_{G/B}(S^{j}\mathfrak{n}^{*}\otimes \bm{k}_{-\mu}))\rangle q^{j}=m_{\lambda}^{\mu}(q)   \]
In particular, when $G=GL_{N}(\bm{k})$, $m_{\lambda}^{\mu}(q)$ coincides with the Kostka polynomials $K_{\lambda,\mu}(q)$ indexed by partitions $\lambda$ and $\mu$ .
Therefore the coefficients of Kostka polynomials are non-negative, which is a well known result in combinatorics.

Kostka polynomials $K_{\bm{\lambda},\bm{\mu}}^{\pm}(q)$ indexed by $r$-multipartitions $\bm{\lambda},\bm{\mu}$ introduced by Shoji\cite{shoji04}, are a generalization of Kostka polynomials.
When $r=2$, Achar and Henderson\cite{AH08} proved that the IC-stalks of the enhanced nilpotent cone are encoded by the Kostka polynomials, thus showing that the coefficients are also non-negative. It is tempting to pursue similar results for arbitrary $r$.

Finkelberg and Ionov \cite{FI} proposed multivariable  Kostka polynomials $K_{\bm{\lambda},\bm{\mu}}(q_{1},\dots,q_{r})$ indexed by generalized $r$-multipartitions $\bm{\lambda}$ and $\bm{\mu}$ for $r\geqslant 2$.
Their construction is a generalization of Lusztig's $m_{\lambda}^{\mu}(q)$ in type $A$.
They conjectured that $K_{\bm{\lambda},\bm{\mu}}(q):=K_{\bm{\lambda},\bm{\mu}}(q,\dots,q)$ coincides with Shoji's version of Kostka functions $K_{\bm{\lambda},\bm{\mu}}^{-}(q)$, which has been proved by Shoji recently\cite{shoji17}.
Similarly to Brylinski's work, they gave a coherent realization of these polynomials by constructing a certain vector bundle $\pi_{r}:\mathcal{T}_{r}^{*}\mathcal{B}_{N}^{r}\rightarrow \mathcal{B}_{N}^{r}$ over the the flag variety $\mathcal{B}_{N}^{r}=\mathcal{B}_{N}\times \cdots \times \mathcal{B}_{N}$ ($r$-factors) of $GL_{N}^{r}=GL_{N}\times \cdots \times GL_{N}$ ($r$-factors).
They showed that the $\mathbb{Z}/r\mathbb{Z}$-graded Euler characteristic equals to  $\chi(\mathcal{T}_{r}^{*}\mathcal{B}_{N}^{r}, \pi_{r}^{*}\mathcal{L}_{\mathcal{B}_{N}^{r}}(\bm{k}_{\bm{-\mu}}))=\sum_{\bm{\lambda}\geqslant \bm{\mu}}K_{\bm{\lambda},\bm{\mu}}(q_{1},\cdots,q_{r})\chi_{\bm{\lambda}}^{*}$.
Here $\chi_{\bm{\lambda}}^{*}$ is the character of $V_{\bm{\lambda}}^{*}$, the dual of irreducible representation of $GL_{N}^{r}$ with highest weight $\bm{\lambda}$.
To prove the coefficients of $K_{\bm{\lambda},\bm{\mu}}(q_{1},\cdots,q_{r})$ are non-negative, they conjectured that $H^{>0}(\mathcal{T}_{r}^{*}\mathcal{B}_{N}^{r}, \pi_{r}^{*}\mathcal{L}_{\mathcal{B}_{N}^{r}}(\bm{k}_{\bm{-\mu}}))=0$ for any dominant $\bm{\mu}$.
When $\bm{k}$ is of characteristic $p>0$, they proved this vanishing theorem for dominant and regular $\bm{\mu}$ by showing that $\mathcal{T}_{r}^{*}\mathcal{B}_{N}^{r}$ is Frobenius split.

In the characterisitc $0$ case, Finkelberg-Ionov's construction is actually a particular case
of Panyushev's generalization of Brylinski's result to homogeneous vector bundles over the flag
variety associated to arbitrary $B$-modules \cite{Pan10}.
Following Broer\cite{broer2}, Panyushev also gave a sufficient condition for the higher cohomology vanishing of line bundles and hence for the nonnegativity of coefficients of his version of generalized Kostka polynomials.
The key step is an application of Kempf's version of Grauert-Riemenschneider vanishing theorem\cite{kempf76}. We refer to \cite{Li17} for a description of Panyushev's result in the reductive group case.
However, Panyushev only considered the one-variable polynomial.
In Finkelberg-Ionov's particular case, we have a finer $\mathbb{Z}/r\mathbb{Z}$-grading for the cohomology, so that the result is a polynomial in $r$ variables.

Inspired by \cite{FI}, D. Orr and M. Shimozono \cite{OS17} for any quiver $Q$ and its acyclic
subquiver $\hat{Q}$, defined the quiver Kostka-Shoji polynomials
$K_{\lambda,\mu}^{Q,\hat{Q}}(q_{Q_{1}})$, using Lustig's convolution diagram $Z^{Q,\hat{Q}}$.
They conjectured that the quiver Kostka-Shoji polynomials also have non-negative coefficients.

This paper is organized as follows.
We consider the above vanishing problem in characteristic $0$ case first.
Specifically, we show that Panyushev's sufficient condition holds in this case by considering the collapsing as a certain generalization of Springer's resolution.
For the characteristic $p>0$ case, note that there is a version of Grauert-Riemenschneider vanishing theorem in characteristic $p>0$ for Frobenius split varieties given by Mehta and van der Kallen\cite{MK92}.
Combining with Finkelberg-Ionov's result, we show that a similar proof also works in this case.
For quivers of $ADE$ types and cyclic, we define the generalized quiver Koskta-Shoji polynomials $K^{\mono}_{\bm{\lambda},\bm{\mu}}(q_{Q_{1}})$ associated to desingularizations of orbit closures given by \cite{Rei03,Sch04} and prove their positivity.

\section{Resolutions of Quiver Representations of $\tilde{A}_{r-1}$}
\label{two}

Let $\bm{k}$ be an algebraically closed field of characteristic zero.
Fixing $r\geqslant 2$, let $\bm{Q}=(Q_{0},Q_{1})$ be the cyclic quiver of type $\tilde{A}_{r-1}$, with vertices $Q_{0}=\{\mathbb{Z}/r\mathbb{Z}\}$ and directed edges $Q_{1}=\{a_{i}: i\rightarrow i+1\ |\ i\in\mathbb{Z}/r\mathbb{Z}\}$.

A finite dimensional representation $(\bm{V},\bm{x})$ of $\bm{Q}$ with dimension vector $\bm{d}=(d_{0},\dots,d_{r-1})$ is a collection $(\bm{V}=\oplus_{i\in \mathbb{Z}/r\mathbb{Z}}\ V_{i}, \bm{x}=\oplus_{i\in \mathbb{Z}/r\mathbb{Z}}\ x_{i})$ of vector spaces over $\bm{k}$ and linear maps with $\text{dim}V_{i}=d_{i}$ and $x_{i}\in \text{Hom}_{\bm{k}}(V_{i},V_{i+1})$.
We will thus define the quiver representation space with dimension vector $\bm{d}$ as
\[\text{Rep}(\bm{Q},\bm{d}):=\bigoplus_{i\in \mathbb{Z}/r\mathbb{Z}}\text{Hom}_{\bm{k}}(\bm{k}^{d_{i}},\bm{k}^{d_{i+1}})\]
There is a natural $G_{\bm{d}}=\prod_{i\in \mathbb{Z}/r\mathbb{Z}}GL_{d_{i}}(\bm{k})$ action on $\text{Rep}(\bm{Q},\bm{d})$ by base changes, namely the action is given by
\[(g_{0},g_{1},\dots,g_{r-1})\cdot(x_{0},x_{1},\dots,x_{r-1})\mapsto(g_{0}x_{0}g_{1}^{-1},g_{1}x_{1}g_{2}^{-1},\dots,g_{r-1}x_{r-1}g_{0}^{-1})\]
Hence the $G_{\bm{d}}$-orbits of $\text{Rep}(\bm{Q},\bm{d})$  are in one-to-one correspondence with the isomorphism classes of quiver representations.

A subrepresentation $(\bm{U},\bm{y})$ of $(\bm{V},\bm{x})$ is a collection of graded subspaces $\bm{U}=\oplus_{i\in\mathbb{Z}/r\mathbb{Z}}\ U_{i}$ of $\bm{V}$ with $U_{i}\subset V_{i}$ and $y_{i}=x_{i}|_{U_{i}}$ for each $i\in \mathbb{Z}/r\mathbb{Z}$.
A filtration $\mathcal{F}^{\bullet}(\bm{V},\bm{x})$ by subrepresentations $0=(\bm{U}^{0},\bm{y}^{0})\subset\cdots\subset(\bm{U}^{m},\bm{y}^{m})=(\bm{V},\bm{x})$ is a filtration of $\bm{V}$ by graded subspaces $\mathcal{F}^{\bullet}\bm{V}=\{\bm{U}^{j}\}_{j=0,1,\dots,m}$ with $\bm{x}$ stabilizing each $\bm{U}^{j}$.
Let $\bm{d}^{j}$ be the dimension vector of $\bm{U}^{j}$. We say the above filtration $\mathcal{F}^{\bullet}\bm{V}$ is of type $\bm{D}=(\bm{d}^{0},\dots,\bm{d}^{m})\in \mathbb{N}^{rm}$ with $\bm{d}^{0}=0$, $\bm{d}^{m}=\bm{d}$ and $(d_{i}^{0},\dots,d_{i}^{m})$ an increasing sequence for each $i\in\mathbb{Z}/r\mathbb{Z}$.
We denote by $\mathcal{F}^{\bullet}_{\bm{d},\bm{D}}$ the space of all filtrations of graded spaces of type $\bm{d}$ with type $\bm{D}$.
It is rather important to study the interplay of the flags of representation space and the quiver representation acting on it.
The following definition is due to Lusztig\cite{Lus91} as
a generalization of the Grothendieck-Springer resolution $\phi: \tilde{\mathfrak{g}}\rightarrow \mathfrak{g}$ for type $A$\cite{CG10}:
\begin{Def}
A resolution of quiver representation space $\text{Rep}(\bm{Q},\bm{d})$ associated to a filtration of type $\bm{D}$, denoted by $\widetilde{\text{Rep}}(\bm{Q},\bm{d},\bm{D})$, is the variety
\[\{(\bm{x},\bm{y}) \in \text{Rep}(\bm{Q},\bm{d})\times \mathcal{F}^{\bullet}_{\bm{d},\bm{D}}\ |\ \bm{y}=\{0=\bm{U}^{0}\subset\dots\subset\bm{U}^{m}\},\bm{x}(\bm{U}^{j})\subset \bm{U}^{j}\
\operatorname{for\ any}\  j=0,1,\dots,m\}.\]
The first and second projection give rise to a diagram
\[\xymatrix{
\  & \widetilde{\text{Rep}}(\bm{Q},\bm{d},\bm{D})\ar[dl]_{p_{1}}\ar[dr]^{p_{2}} & \ \\
\text{Rep}(\bm{Q},\bm{d})  & \  & \ \ \ \ \mathcal{F}^{\bullet}_{\bm{d},\bm{D}}\ \ \ \ \ \
}\]
where for $\bm{x}\in \text{Rep}(\bm{Q},\bm{d})$, $p_{1}^{-1}(\bm{x})$ is the space of filtrations
stabilized by $\bm{x}$; for $\bm{y}\in \mathcal{F}_{\bm{d},\bm{D}}^{\bullet}$, $p_{2}^{-1}(\bm{y})$ is the space of quiver representations that stabilize $\bm{y}$.
\end{Def}

From now on we fix the dimension vector $\bm{d}_{0}=(N,\dots,N)$, i.e., each $V_{i}=\bm{k}^{N}$.
If we choose $\bm{D}$ to be $((0,\dots,0),(1,\dots,1),(2,\dots,2),\dots,(N,\dots,N))$, then for each $i\in \mathbb{Z}/r\mathbb{Z}$ we have a complete flag $0=U_{i}^{0}\subset U_{i}^{1}\subset\dots \subset U_{i}^{N}=V_{i}$ of $\bm{k}^{N}$.
In this case, $\mathcal{F}_{\bm{d}_{0},\bm{D}}^{\bullet}$ is isomorphic to $\mathcal{B}_{N}^{r}$,
i.e., the $r$-th Cartesian power of the flag variety $\mathcal{B}_{N}$ of $GL_{N}(\bm{k})$.
We say that this filtration is of standard type $\bm{D}_{0}$, hence $\mathcal{F}_{\bm{d}_{0},\bm{D}_{0}}^{\bullet}=\mathcal{B}_{N}^{r}$.
By fixing a basis $\{\upsilon_{i}^{1},\dots,\upsilon_{i}^{N}\}$ for each $V_{i}$, $\text{Rep}(\bm{Q},\bm{d})$ is identified with the direct sum of $r$ copies of the space of $N\times N$ matrices $M_{N}$.
We identify $M_{N}$ with $\mathfrak{g}$, the lie algebra of $GL_{N}(\bm{k})$.
Let $\bm{y}_{0}=\{0=\bm{U}^{0}\subset\dots\subset\bm{U}^{N}=\bm{V}\}$ be the standard flag with $U_{i}^{j}=\text{Span}\{\upsilon_{i}^{1},\dots,\upsilon_{i}^{j}\}$ for any $0\leqslant j\leqslant N$.
A quiver representation $\bm{x}\in \text{Rep}(\bm{Q},\bm{d})$ stabilizes $\bm{y}_{0}$ if and only if $x_{i}\in \mathfrak{b}$, that is, $x_i$ is
an upper triangular matrix for each $i\in \mathbb{Z}/r\mathbb{Z}$. The group
$G_{\bm{d}_{0}}=\prod_{i\in \mathbb{Z}/r\mathbb{Z}}GL_{N}(\bm{k})$ acts on $\widetilde{\text{Rep}}(\bm{Q},\bm{d}_{0},\bm{D}_{0})$ by base changes.
The following realization of $\widetilde{\text{Rep}}(\bm{Q},\bm{d}_{0},\bm{D}_{0})$
is similar to~\cite[Corollary~3.1.33]{CG10}.
We denote $\bm{\mathfrak{b}_{r}}:=\mathfrak{b}^{\prime}_{0}\oplus\cdots\oplus\mathfrak{b}^{\prime}_{r-1}$ ($r$ summands), and $\bm{\mathfrak{g}_{r}}:=\mathfrak{g}^{\prime}_{0}\oplus\cdots\oplus\mathfrak{g}_{r-1}^{\prime}$ ($r$ summands).
Here, we add a prime to remind that they live in $\text{Hom}_{\bm{k}}(V_{i},V_{i+1})$ rather than $\text{Hom}_{\bm{k}}(V_{i},V_{i})$.

\begin{Prop}\label{bundle}
The projection $p_{2}:\widetilde{\text{Rep}}(\bm{Q},\bm{d}_{0},\bm{D}_{0})\rightarrow \mathcal{F}^{\bullet}_{\bm{d}_{0},\bm{D}_{0}}=\mathcal{B}_{N}^{r}$ makes $\widetilde{\text{Rep}}(\bm{Q},\bm{d}_{0},\bm{D}_{0})$ a $GL_{N}^{r}$-equivariant vector bundle over $\mathcal{B}_{N}^{r}$ with fiber $\bm{\mathfrak{b}_{r}}$.
For $\bm{g}\in GL_{N}^{r}$, the assignment
\[(g_{0},\dots,g_{r-1},x_{0},\dots,x_{r-1})\mapsto (g_{0}x_{0}g_{1}^{-1},\dots,g_{r-1}x_{r-1}g_{0}^{-1},\bm{g}\cdot\bm{y}_{0})\]
gives a $GL_{N}^{r}$-equivariant isomorphism $GL_{N}^{r}\times^{B_{N}^{r}}\bm{\mathfrak{b}_{r}}\simeq\widetilde{\text{Rep}}(\bm{Q},\bm{d}_{0},\bm{D}_{0})$.
\end{Prop}

We denote by $\mathfrak{T}_{r}^{*}\mathcal{B}_{N}^{r}$ the vector bundle $GL_{N}^{r}\times^{B_{N}^{r}}\bm{\mathfrak{b}_{r}}$ over the flag variety $\mathcal{B}_{N}^{r}$.
From this point of view, it is evident that $\widetilde{\text{Rep}}(\bm{Q},\bm{d}_{0},\bm{D}_{0})$ is smooth.
For this particular case, we shall denote the first projection by $\psi_{r}$, the second projection by $\pi_{r}$.
Under the above identification, $\psi_{r}:\mathfrak{T}_{r}^{*}\mathcal{B}_{N}^{r}\rightarrow \bm{\mathfrak{g}_{r}}$ is given by
\[(g_{0},\dots,g_{r-1},x_{0},\dots,x_{r-1})\mapsto (g_{0}x_{0}g_{1}^{-1},\dots,g_{r-1}x_{r-1}g_{0}^{-1}).\]
It is called {\em collapsing} following Kempf. Clearly it is proper and surjective. Moreover, we have

\begin{Lem}\label{relative dimension}
$\psi_{r}:\mathfrak{T}_{r}^{*}\mathcal{B}_{N}^{r}\rightarrow \bm{\mathfrak{g}_{r}}$ is of relative dimension zero, i.e., it is generically finite.
\end{Lem}

\begin{proof}
It suffices to prove that there is a dense open subset of $\bm{\mathfrak{g}_{r}}$ such that for any element of this set, its inverse image is a finite set.
Set
\[\bm{\mathfrak{g}_{r}}^{sr}:=\{(x_{0},\dots,x_{r-1})\in \bm{\mathfrak{g}_{r}}\ |\ x_{1}\cdot x_{2}\cdots x_{0}\  \text{is a semisimple regular matrix}\}\]
Similarly to~\cite[Lemma~3.1.5]{CG10}, $\bm{\mathfrak{g}_{r}}^{sr}$ is a dense open subset of $\bm{\mathfrak{g}_{r}}$.

Choose any $\bm{x}=(x_{0},\dots,x_{r-1})\in \bm{\mathfrak{g}_{r}}^{sr}$.
For any $i\in \mathbb{Z}/r\mathbb{Z}$, $f_{i}=x_{i+1}\cdot x_{i+2}\cdots x_{i}: V_{i}\rightarrow V_{i}$ is a linear transformation of $V_{i}$, and the condition says $f_{0}$ is semisimple and regular.
Since for two $N\times N$ matrices $A$ and $B$, $A\cdot B$ and $B\cdot A$ have the same eigenvalues, we see that $f_{i}$ is also semisimple and regular for any $i\in\mathbb{Z}/r\mathbb{Z}$.
Given $\bm{y}\in \psi_{r}^{-1}(\bm{x})$, we see that $f_{i}$ stablizes the filtration of $V_{i}$ for each $i\in\mathbb{Z}/r\mathbb{Z}$.
Hence the number of filtrations of $\bm{V}$ that are stablized by $\bm{x}$ is not bigger than $(n!)^{r}$.
\end{proof}

Now, Finkelberg-Ionov's bundle $\mathcal{T}_{r}^{*}\mathcal{B}_{N}^{r}$ is constructed as follows.
Let $\bm{D}_{1}=(\bm{d}^{0},\bm{d}^{1},\dots,\bm{d}^{2N})$ be the type of filtration with $\bm{d}^{2k}_{\ell}=k$ if $1\leqslant k\leqslant N$ and $0\leqslant \ell\leqslant r-1$; $\bm{d}_{\ell}^{2k-1}=k-1$ if $1\leqslant k\leqslant N$ and $0\leqslant \ell\leqslant r-2$; and $\bm{d}^{2k-1}_{r-1}=k$ for $1\leqslant k\leqslant N$.
We denote by  $\mathcal{T}_{r}^{*}\mathcal{B}_{N}^{r}$ the resolution variety $\widetilde{\text{Rep}}(\bm{Q},\bm{d_{0}},\bm{D}_{1})$. In this case  $\mathcal{F}^{\bullet}_{\bm{d}_{0},\bm{D}_{1}}$ is also the flag variety $\mathcal{B}_{N}^{r}$.
This construction is actually nothing but the realization of Lusztig's convolution diagram given by \cite{FI}.
Let $\mathfrak{n}^{\prime}_{i}$ be the space of strictly upper triangular matrices in $\mathfrak{g}_{i}^{\prime}$ ($i\in\mathbb{Z}/r\mathbb{Z}$).
We denote $\bm{\mathfrak{n}_{r}}:=\mathfrak{b}_{0}^{\prime}\oplus\cdots\oplus\mathfrak{b}^{\prime}_{r-2}\oplus\mathfrak{n}^{\prime}_{r-1}$ ($r$ summands).
Similarly to Proposition~\ref{bundle}, $\mathcal{T}_{r}^{*}\mathcal{B}_{N}^{r}\simeq GL_{N}^{r}\times^{B_{N}^{r}}\bm{\mathfrak{n}_{r}}$, which is a vector subbundle of $\mathfrak{T}_{r}^{*}\mathcal{B}_{N}^{r}$ over $\mathcal{B}_{N}^{r}$.

We denote by $\varphi_{r}: \mathcal{T}_{r}^{*}\mathcal{B}_{N}^{r}\rightarrow \bm{\mathfrak{g}_{r}}$ the collapsing.
Its image $\mathcal{N}_{r}$ is a closed $GL_{N}^{r}(\bm{k})$-stable subvariety of $\bm{\mathfrak{g}_{r}}$.
As a generalization of Springer resolution, we have the following result due to
G.~Lusztig, see~\cite[Lemma~1.6]{Lus91}.

\begin{Prop}\label{birational}
$\varphi_{r}: \mathcal{T}_{r}^{*}\mathcal{B}_{N}^{r}\rightarrow \mathcal{N}_{r}$ is a resolution of singularities of $\mathcal{N}_{r}$, i.e., it is a proper surjective birational morphism.
In particular, it is of relative dimension zero.
\end{Prop}

\begin{proof}
  For $\bm{x}=(x_{0},\dots,x_{r-1})\in \mathcal{N}_{r}$, we still denote by $f_{i}$
  the composition $x_{i+1}\cdot x_{i+2}\cdots\cdot x_{i}: V_{i}\rightarrow V_{i}$ ($i\in \mathbb{Z}/r\mathbb{Z}$).
In this case, for any $\bm{y}\in\varphi^{-1}(\bm{x})$, for the filtration $0=U_{i}^{0}\subset U_{i}^{1}\subset\cdots \subset U_{i}^{N}=V_{i}$ of $V_{i}$, we have $f_{i}(U_{i}^{k})\subset U_{i}^{k-1}$ for $1\leqslant k\leqslant N$.
Thus each $f_{i}$ is a nilpotent matrix.
We set
\[\mathcal{N}_{r}^{\text{reg}}:=\{(x_{0},\dots,x_{r-1})\in\mathcal{N}_{r}\ |\ f_{i}  \ \text{is regular nilpotent for any }i\in\mathbb{Z}/r\mathbb{Z}\}\]
Since the non-regularity property is defined by vanishing of certain polynomials,  $\mathcal{N}_{r}^{\text{reg}}$ is a dense open subset of $\mathcal{N}_{r}$.
By~\cite[Proposition~3.2.13]{CG10}, there is one and only one flag that is fixed by a regular nilpotent matrix.
Thus for any $\bm{x}\in\mathcal{N}_{r}^{\text{reg}}$, the fiber $\varphi_{r}^{-1}(\bm{x})$ consists of exactly one element.
\end{proof}

\section{Higher Cohomology Vanishing}
In this section, we prove our main result, i.e., the higher cohomology vanishing of line bundles over $\mathcal{T}_{r}^{*}\mathcal{B}_{N}^{r}$ in the characteristic $0$ case.
The same proof also applies to $\mathfrak{T}_{r}^{*}\mathcal{B}_{N}^{r}$.

We start by recalling some facts about $GL_{N}(\bm{k})$ and fixing some notations.
Let $G=GL_{N}(\bm{k})$, and let $T\subset G$ be the subgroup of diagonal matrices. Denote by $\epsilon_{j}\in X(T)$ the character of $T$ such that $\epsilon_{j}(t_{1},\dots,t_{N})=t_{j}$ for each $1\leqslant j\leqslant N$.
Then $\{\epsilon_{j}\}_{1\leqslant j\leqslant N}$ is a standard basis for $X(T)$.
In this way, $X(T)\simeq\mathbb{Z}^{N}$.
Let $B\subset G$ be the Borel subgroup of upper triangular matrices that corresponds to the positive roots $\{\epsilon_{j}-\epsilon_{k}\}_{1\leqslant j< k\leqslant N}$. Let $\mathcal{B}=G/B$ be the flag variety.
The set of dominant weights is $X^{+}(T)\simeq \{\lambda=(\lambda_{1}\geqslant \lambda_{2}\geqslant\dots\geqslant \lambda_{N})\in \mathbb{Z}^{N}\}$, i.e., the set $\mathcal{P}_{N}$ of generalized partitions of length $N$.
For $G=GL_{N}^{r}$, the weight lattice is $X(T^{r})\simeq \mathbb{Z}^{rN}$. The standard basis is $\{\epsilon_{i,j}\}_{0\leqslant i\leqslant r-1,1\leqslant j\leqslant N}$ ordered as $\{\epsilon_{0,1},\epsilon_{0,2},\dots,\epsilon_{0,N},\epsilon_{1,1},\dots,\epsilon_{r-1,N}\}$.
The set of dominant weights is $X^{+}(T^{r})\simeq\{\bm{\lambda}=(\lambda^{0},\cdots,\lambda^{r-1})\in \mathbb{Z}^{rN}|\ \lambda^{i}\in \mathcal{P}_{N}\  \text{for any}\ i\in \mathbb{Z}/r\mathbb{Z}\}$, i.e., the set $\mathcal{P}_{N}^{r}$ of generalized $r$-multipartitions.
If $\lambda\in \mathcal{P}_{N}$, then $H^{0}(\mathcal{B},\mathcal{L}_{\mathcal{B}}(\bm{k}_{-\lambda}))$ is isomorphic to the irreducible representation of $G$ with lowest weight $-\lambda$, i.e., the dual of irreducible representation $V_{\lambda}$ of $G$ with highest weight $\lambda$.
For $\bm{\lambda}\in \mathcal{P}_{N}^{r}$, $H^{0}(\mathcal{B}_{N}^{r},\mathcal{L}_{\mathcal{B}_{N}^{r}}(\bm{k}_{\bm{-\lambda}}))\simeq V_{\bm{\lambda}}^{*}=V_{\lambda^{0}}^{*}\otimes\cdots\otimes V_{\lambda^{r-1}}^{*}$.




Similarly to~\cite[Theorem 2.2]{broer2}, we consider $U:=GL_{N}^{r}\times^{B_{N}^{r}}(\bm{\mathfrak{n}_{r}}\times\bm{k}_{\bm{\lambda}})$ for any dominant weight $\bm{\lambda}\in \mathcal{P}_{N}^{r}$.
It can be viewed as a line bundle over $\mathcal{T}_{r}^{*}\mathcal{B}_{N}^{r}$.
We have a natural map $GL_{N}^{r}\times^{B_{N}^{r}}\bm{k}_{\bm{\lambda}}\rightarrow V_{\bm{\lambda}}$ given by $(\Pi_{i=0}^{r-1}g_{i},c_{\bm{\lambda}})\mapsto \Pi_{i=0}^{r-1}g_{i}\cdot\upsilon_{\bm{\lambda}}$, where $\upsilon_{\bm{\lambda}}$ is the highest weight vector.
Let $\phi:U\rightarrow \mathcal{N}_{r}\times V_{\bm{\lambda}}$ be the morphism:
\[
\phi:(g_{0},\dots,g_{r-1},x_{0},\dots,x_{r-1},c_{\bm{\lambda}})\mapsto (g_{0}x_{0}g_{1}^{-1},\dots,g_{r-1}x_{r-1}g_{0}^{-1},\Pi_{i=0}^{r-1}g_{i}\cdot\upsilon_{\bm{\lambda}})
\]
By Proposition \ref{birational}, $\phi:U\rightarrow\text{Im}(\phi)$ is proper and biraional.

The following lemma is similar to~\cite[Lemma 4.3]{FI}.

\begin{Lem}\label{canonical bundle}
Let $\bm{\kappa}:=\sum_{j=0}^{N}(\epsilon_{0,j}-\epsilon_{r-1,j})$.
We have $K(U)\simeq \mathcal{L}_{U}(\bm{k}_{-\bm{\lambda}-\bm{\kappa}})$, where we denote by $K(U)$ the canonical line bundle of $U$.
\end{Lem}

\begin{proof}
Since $U$ is the total space of $GL_{N}^{r}$-equivariant vector bundle over $\mathcal{B}_{N}^{r}$, the canonical line bundle $K(U)$ is a pull back line bundle from $\mathcal{B}_{N}^{r}$.
We need to determine its weight.
The top exterior power of the cotangent space of $U$ at $(e,z)$ ($e$ is the identity in $GL_{N}^{r}$, and $z\in \bm{\mathfrak{n}_{r}}\times \bm{k}_{\bm{\lambda}}$) is
\[\bigotimes_{i=0}^{r-1}\wedge^{\text{top}}(\mathfrak{g}_{i}/\mathfrak{b}_{i})^{*}\otimes\bigotimes_{i=0}^{r-1}\wedge^{\text{top}}(\mathfrak{n}^{\prime}_{i})^{*}\otimes\bm{k}_{-\bm{\lambda}-\bm{\kappa}}\]
For a one-dimensional representation $M$ of $T^{r}$, we shall denote $wt(M)$ its weight.
The weights of $(\mathfrak{g}_{i}/\mathfrak{b}_{i})^{*}$ are all positive roots, hence $wt(\otimes_{i=0}^{r-1}\wedge^{\text{top}}(\mathfrak{g}_{i}/\mathfrak{b}_{i})^{*})=\sum_{i=0}^{r-1}\sum_{1\leqslant j<k\leqslant N}(\epsilon_{i,j}-\epsilon_{i,k})$.
Let $E_{i,jk}$ ($1\leqslant j<k\leqslant N$) be the natural basis of $\mathfrak{n}^{\prime}_{i}$, then $wt(E_{i,jk})=\epsilon_{i,j}-\epsilon_{i+1,k}$.
Hence $wt(\otimes_{i=0}^{r-1}\wedge^{\text{top}}(\mathfrak{n}_{i}^{\prime})^{*})=\sum_{i=0}^{r-1}\sum_{1\leqslant j< k\leqslant N}(\epsilon_{i,k}-\epsilon_{i+1,j})=\sum_{i=0}^{r-1}\sum_{1\leqslant j<k\leqslant N}(\epsilon_{i,k}-\epsilon_{i,j})$.
Thus the first two terms cancel each other and leave $\bm{k}_{-\bm{\lambda}-\bm{\kappa}}$ alone.
\end{proof}

Now our main result follows by Broer-Brylinski paradigm.

\begin{Thm}
For any $\bm{\lambda}\in \mathcal{P}_{N}^{r}$ we have $H^{>0}(\mathcal{T}_{r}^{*}\mathcal{B}_{N}^{r},\mathcal{L}_{\mathcal{T}_{r}^{*}\mathcal{B}_{N}^{r}}(\bm{k}_{\bm{-\lambda}}))=0$.
\end{Thm}

\begin{proof}
  By Proposition \ref{birational} and Lemma \ref{canonical bundle}, we can argue similarly
  to\cite[Theorem~2.2]{broer2} or we can apply~\cite[Theorem~3.1.(ii)]{Pan10}. The key step is
  to prove $H^{>0}(U,K(U))=0$.
\end{proof}

\begin{Rk}
{\em By Lemma \ref{relative dimension}, the same proof gives us the higher cohomology vanishing for the same line bundles over $\mathfrak{T}_{r}^{*}\mathcal{B}_{N}^{r}$.}
\end{Rk}

As a corollary, we have
\begin{Cor}
  For any $\bm{\lambda}\geqslant\bm{\mu}\in \mathcal{P}_{N}^{r}$, Finkelberg-Ionov's multivariable version of Kostka polynomials have nonnegative coefficients:
  $K_{\bm{\lambda},\bm{\mu}}(q_{1},\dots,q_{r})\in \mathbb{N}[q_{1},\dots,q_{r}]$.
  In particular, Shoji's negative version of Kostka polynomials have nonnegative coefficients:
  $K^{-}_{\bm{\lambda},\bm{\mu}}(q)\in \mathbb{N}[q]$ for any $\bm{\lambda}\geqslant\bm{\mu}\in\mathcal{P}_{N}^{r}$.

\end{Cor}

\section{Characteristic $p>0$ case}

In this section, we assume $\bm{k}$ to be algebraically closed of characteristic $p>0$ and prove the above higher cohomology vanishing theorem in this case.

All the results of Section~\ref{two} are still valid. We refer to \cite{Springer66} for the proof of analogous results needed in the characteristic $p>0$ case.
We can still construct the variety $U$ and consider the morphism $\phi: U\rightarrow \mathcal{N}_{r}\times V_{\bm{\lambda}}$, except that here $V_{\bm{\lambda}}$ is the Weyl module of $GL_{N}^{r}$ with highest weight $\bm{\lambda}$.
The collapsing $\phi: U\rightarrow \text{Im}(\phi)$ is proper and birational and Lemma \ref{canonical bundle} is also true.
All we have to do is to prove $H^{>0}(U,K(U))=0$.
We will need the following version of Grauert-Riemenschneider vanishing theorem in characteristic $p>0$ case due to Mehta and van der Kallen\cite{MK92}.

\begin{Thm}\label{MK}
  Let $\pi:X\rightarrow Y$ be a proper birational morphism of varieties in characteristic $p>0$ such that:
  
  1. $X$ is non-singular and there is $\sigma\in H^{0}(X,K(X)^{-1})$ such that $\sigma^{p-1}$ splits X.
  
  2. $D=\text{div}(\sigma)$ contains the exceptional locus of $\pi$ set theoretically.
  
Then $R^{i}\pi_{*}K(X)=0$ for $i>0$.
\end{Thm}

Finkelberg and Ionov proved that the section $\sigma=f^{p-1}\varpi^{1-p}\in\Gamma(\mathcal{T}^{*}_{r}\mathcal{B}_{N}^{r},\omega^{1-p})$ splits $\mathcal{T}_{r}^{*}\mathcal{B}_{N}^{r}$ in \cite{FI}.
Here $\varpi$ is a local section of the canonical bundle $\omega$ of $\mathcal{T}_{r}^{*}\mathcal{B}_{N}^{r}$, $f$ is a regular function on $\mathcal{T}_{r}^{*}\mathcal{B}_{N}^{r}$ defined by
\[ f(g_{0},\dots,g_{r-1},x_{0},\dots,x_{r-1}):=\prod_{s=0}^{r-2}\prod_{j=1}^{N}\Delta_{j}(g_{s}x_{s}g_{s+1}^{-1})\cdot\prod_{j=1}^{N-1}\Delta_{j}(g_{r-1}x_{r-1}g_{0}^{-1}),\]
where $\Delta_{j}$ denotes the $j$-th principal minor in the upper left corner of a matrix.
Their construction is a generalization of results given in \cite{MK1992}.

\begin{Prop}\label{splitting}
$U$ is Frobenius split.
\end{Prop}

\begin{proof}
$U$ is the line bundle over $\mathcal{T}_{r}^{*}\mathcal{B}_{N}^{r}$ with sheaf of sections $\mathcal{L}_{\mathcal{T}_{r}^{*}\mathcal{B}_{N}^{r}}(\bm{k}_{\bm{\lambda}})$.
By~\cite[Lemma~1.1.11]{BK05}, we get that $U$ is Frobenius split.
The splitting $\sigma^{\prime}: F_{*}\mathcal{O}_{U}\rightarrow \mathcal{O}_{U}$ is given by
\[\sigma^{\prime}(h\gamma^{t})=
\begin{cases}
\sigma(h)\gamma^{t/p} & \text{if}\  t\equiv 0\pmod{p}\\
0 & \text{otherwise},

\end{cases}\]
where $h$ is a local section of $\mathcal{O}_{\mathcal{T}_{r}^{*}\mathcal{B}_{N}^{r}}$, and $\gamma$ is a local trivialization of the dual bundle $\mathcal{L}_{\mathcal{T}_{r}^{*}\mathcal{B}_{N}^{r}}(\bm{k}_{-\bm{\lambda}})$.
\end{proof}

The following useful result is~\cite[Lemma~3.3]{MK1992}.

\begin{Lem}\label{useful}
Let $g$ be a unipotent lower triangular matrix and let $M\in \mathfrak{g}$ be such that $M_{\leqslant r,\leqslant N-r}=0$ for some integer $1\leqslant r\leqslant N$. Then
\[\det((gMg^{-1})_{\leqslant r,\leqslant r})=\det(M_{\leqslant r,>N-r})\cdot\det((g^{-1})_{>N-r,\leqslant r})\]
\end{Lem}

We prove that $\phi: U\rightarrow \text{Im}(\phi)$ satisfies the second condition in Theorem \ref{MK}, the proof being similar to that of~\cite[Proposition~4.5]{MK1992}.

\begin{Prop}
The divisor $D=\text{div}(\sigma^{\prime})$ contains the exceptional locus of $\phi: U\rightarrow \text{Im}(\phi)$.
\end{Prop}

\begin{proof}
By Proposition \ref{splitting}, we only need to prove that $\text{div}(\sigma)$ contains the exceptional locus $\mathcal{N}_{r}\backslash\mathcal{N}_{r}^{\text{reg}}$. As in~\cite[Lemma~4.7]{FI}, we consider the open cell $X:=X_{N}^{r}\times\bm{\mathfrak{n}_{r}}\subset \mathcal{T}_{r}^{*}\mathcal{B}_{N}^{r}$, where $X_{N}\subset\mathcal{B}_{N}$ is the open orbit of the strictly lower triangular subgroup $U_{N}^{-}\subset GL_{N}$. Identifying $X$ with $\hat{X}:=(U_{N}^{-}\times\mathfrak{b}_{N})^{r-1}\times(U_{N}^{-}\times\mathfrak{u}_{N})$, we have $f|_{X}=\hat{f}|_{\hat{X}}$ where
\[\hat{f}(g_{0},\dots,g_{r-1},x_{0},\dots,x_{r-1}):=\prod_{s=0}^{r-2}\prod_{j=1}^{N}\Delta_{j}(g_{s}x_{s}g_{s}^{-1})\cdot\prod_{j=1}^{N-1}\Delta_{j}(g_{r-1}x_{r-1}g_{r-1}^{-1}),\]
where $g_{0\leqslant s\leqslant r-1}$ are all lower triangular unipotent matices. Since $X$ is dense, we can restrict to this set. We have to show that if $\hat{f}(g_{0},\dots,g_{r-1},x_{0},\dots,x_{r-1})\neq 0$, then $(g_{0}x_{0}g_{1}^{-1},\dots,g_{r-1}x_{r-1}g_{0}^{-1})$ belongs to $\mathcal{N}_{r}^{\text{reg}}$. This follows easily by computations using Lemma \ref{useful}.
\end{proof}

\section{Positivity of the quiver Kostka-Shoji polynomials for cyclic and ADE quivers}

\subsection{Quiver Representations} Let $\bmk$ be an algebraically closed field of characteristic $0$. The vector spaces we consider are assumed to be over $\bmk$. A quiver $Q=(Q_{0},Q_{1})$ is a directed graph with vertex set $Q_{0}$ and arrow set $Q_{1}$. For an arrow $j\in Q_{1}$, we
denote by $h(j), t(j)\in Q_{0}$ its head and tail. We assume that there is at most one arrow from each vertex to another. In this section, we will deal only with the cyclic quivers and finite
ADE quivers. Let $V=\oplus_{i\in Q_{0}}V_{i}$ be a $Q_{0}$-graded vector spaces with dimension vector $d=\underline{\text{dim}} V=(d_{i})_{i\in Q_{0}}$. We denote by $E_{d}=\prod_{j\in Q_{1}}\text{Hom}_{\bmk}(V_{h(j)},V_{t(j)})$ the representation space on $V$. The group $G_{d}=\prod_{i\in Q_{0}}GL_{d_{i}}(\bmk)$ acts on $E_{d}$ via $(g_{i})\cdot(x_{j})=(g_{t(j)}x_{j}g_{h(j)}^{-1})$ by base changes.

\subsection{Lustig's iterated convolution diagram}Let $\bm{i}=(i_{1},i_{2},\dots,i_{\nu})\in Q_{0}^{\nu}$, $\bm{a}=(a_{1},a_{2},\dots,a_{\nu})\in \mathbb{N}^{\nu}$. The variety $\mathcal{F}_{\bm{i},\bm{a}}$ of all flags of type $(\bm{i},\bm{a})$ in $V$ consists of filtrations of $Q_{0}$-graded subspaces
$V=\mathcal{F}^{0}\supset\mathcal{F}^{1}\supset\cdots\supset\mathcal{F}^{\nu}$ such that $\mathcal{F}^{k-1}/\mathcal{F}^{k}$ is pure of weight $i_{k}$ and dimension $a_{k}$. Lustig's iterated convolution diagram $X_{\bm{i},\bm{a}}$ is the incidence variety whose points are pairs $(\varphi,\mathcal{F})\in E_{d}\times \mathcal{F}_{\bm{i},\bm{a}}$ such that $\varphi(\mathcal{F}^{k})\subset \mathcal{F}^{k}$ for all $k=0,1,\dots,\nu$. The natural projection $X_{\bm{i},\bm{a}}\rightarrow \mathcal{F}_{\bm{i},\bm{a}}$ makes $X_{\bm{i},\bm{a}}$ a homogeneous bundle with typical fiber $Y_{\bm{i},\bm{a}}$, i.e., by fixing a flag $\mathcal{F}^{*}$, we get an isomorphism $X_{\bm{i},\bm{a}}\simeq G_{d}\times^{P_{\bm{i},\bm{a}}}Y_{\mono}$, where $P_{\mono}$ is the stabilizer of $\mathcal{F}^{*}$ under the $G_{d}$-action. The other projection $\pi_{\mono}: X_{\mono}\rightarrow E_{d}$ is a generalization of the
Grothendieck-Springer morphism, which is called collapsing morphism in the sense of Kempf.

Let $\hat{Q}$ be an acyclic subquiver with the same vertex set, let $V_{n}$ be the standard $GL_{n}$-module, $V=\oplus_{i\in Q_{0}}V_{n}^{(i)}$, and $\mathcal{F}^{Q_{0}}_{n}$ is the $Q_{0}$-product of complete flags in $V_{n}$. Orr and Shimozono \cite{OS17} consider the following instance of convolution diagram $Z^{Q,\hat{Q}}$, which consists of pairs $(\phi_{a},F^{(i)})\in E_{d}\times \mathcal{F}_{n}^{Q_{0}}$ such that
\begin{equation}
\phi_{a}(F_{k}^{(h(a))})\subset
\begin{cases}
F_{k}^{(t(a))} & \text{if } a\in \hat{Q}_{1}\\
F_{k-1}^{(t(a))} & \text{if } a\in Q_{1}\backslash\hat{Q}_{1}.\label{eq1}
\end{cases}
\end{equation}
By \cite[Lemma~1.8]{Lus91},  there does exist a pair $(\mono)$ such that $Z^{Q,\hat{Q}}\simeq X_{\mono}$.
 
 \subsection{Relations with \cite{Pan10}}Let $G$ be a reductive group over $\bmk$ (of characteristic $0$),
 $P$ a parabolic subgroup, $V$ a $G$-stable vector space, and $N$ a $P$-stable subspace.
 Set $Z=G\times^{P}N$ and $\pi:Z\rightarrow G\cdot N$, let $|N|$ be the sum of $T$-weights
 in $N$. For any homogeneous bundle $Z$, Panyushev generalized Brylinski's method and construct the associated polynomial $\mathfrak{m}_{\lambda,N}^{\mu}(q)$, as the transition matrix between the graded Euler characteristic function $\chi_{Z}^{\mu}$ of line bundles on $Z$ and the character of simple modules $\chi_{\lambda}$, which he call generalized
 Kostka-Foulkes polynomials. Using Broer's method, Panyushev proved that if the collapsing
 $G\times^{P}N\rightarrow G\cdot N$ is generically finite (i.e., $\dim Z=\dim G\cdot N$), then
 $\mathfrak{m}_{\lambda,N}^{\mu}(q)$ has non-negative coefficients for any $\lambda\in X^{+}(T)$ and $\mu\gtrdot |N|-|\mathfrak{n}|$. Here $\mathfrak{n}$
 is the nil-radical of $P$, and for $\lambda_{1},\lambda_{2}\in X(T)$, we say
 $\lambda_{1}\gtrdot\lambda_{2}$ when $\lambda_{1}-\lambda_{2}\in X^{+}(T)$, that is
 $\lambda_{1}$ lies in the dominant chamber translated by $\lambda_{2}$. If the collapsing fails to be generically finite, then he proved that the positivity still holds for $\mu\gtrdot \rho_{p}+|N|-|\mathfrak{n}|$, where $\rho_{p}$ is the sum of fundamental weights corresponding to $P$.
 
 Applying Panyushev's construction to $Z^{Q,\hat{Q}}$, we get the
 {\em quiver Kostka-Shoji polynomials} $K_{\bm{\lambda},\bm{\mu}}^{Q,\hat{Q}}(q_{Q_{1}})$ introduced
 by Orr and Shimozono. These are polynomials in $|Q_{1}|$ variables since $Y_{\mono}$ is a
 $({\mathbb C}^\times)^{Q_{1}}$-module. Orr and Shimozono conjecture that these polynomials
 have nonnegative coefficients.
 As an immediate corollary of Panyushev's positivity result
 we obtain
 \begin{Thm}
   \label{51}
   The coefficients of $K_{\bm{\lambda},\bm{\mu}}^{Q,\hat{Q}}(q_{Q_{1}})$ are positive for $\bm{\lambda}\in X^{+}(T^{Q_{0}})$ and $\bm{\mu}\gtrdot \rho_{P_{\mono}}+|Y_{\mono}|-|\mathfrak{n}_{P_{\mono}}|$
   (we identify $Z^{Q,\hat{Q}}$ with $X_{\mono}\simeq G_{d}\times^{P_{\mono}}Y_{\mono}$).
 \end{Thm}
 
 More generally, applying the above construction to an arbitrary Lusztig's convolution
 diagram $X_{\mono}$, we denote the corresponding polynomial by
 $K_{\bm{\lambda},\bm{\mu}}^{\mono}(q_{Q_{1}})$.

 \subsection{Desingularizations in cyclic and finite ADE types} To prove the positivity for quiver Kostka-Shoji polynomials $K_{\bm{\lambda},\bm{\mu}}^{\mono}(q_{Q_{1}})$, it is important to find out when the collapsing morphism $\pi_{\mono}:X_{\mono}\rightarrow \text{\rm{Im}}(\pi_{\mono})$ is generically finite. Here the image of $\pi_{\mono}$ is an orbit closure $\overline{\mathcal{O}_{V}}$ of a representation $V$.
 There are many investigations on desingularization of orbit closures of quiver
 representations, starting from~\cite{Lus90,Lus91}. In this subsection we recall
 Reineke's \cite{Rei03} and
 Schiffmann's \cite{Sch04} results on finite ADE and cyclic quivers.

 For finite ADE quivers, by Gabriel's theorem, the isomorphism classes of indecomposable representations are in bijection with the set $R^{+}$ of positive roots for the Dynkin diagram, given by $M_{\alpha}\leftrightarrow \alpha$ with $\underline{\text{dim}} M_{\alpha}=\alpha$. Here we identify $R^{+}$ with a subset of $\mathbb{N}Q_{0}$ by writing each positive root $\alpha$ as a sum of simple roots. Moreover, the category $\text{\rm{Rep}}(Q)$ is representation directed, namely there exists a total order $\alpha_{1},\dots,\alpha_{m}$ on $R^{+}$ such that
 $\text{\rm{Hom}}(M_{\alpha_a},M_{\alpha_b})=0$ for any $b<a$ and
 $\text{\rm{Ext}}^{1}(M_{\alpha_b},M_{\alpha_a})=0$ for any $b\leqslant a$,
 see~\cite[Proposition~4.12]{Lus90}. This total order induces a total order on the set
 of vertices $Q_0$ such that $i<j$ whenever there is an arrow $i\to j$.
 More generally, folllowing~\cite[Definition~2.1]{Rei03}, we consider a
 directed partition of positive roots $R^{+}=\bigcup_{t=1}^{s}\mathcal{I}_{t}$,
 such that $\text{\rm{Hom}}(M_{\alpha},M_{\beta})=0$ for any
 $\alpha\in\mathcal{I}_{t},\beta\in \mathcal{I}_{u},u<t$, and
 $\text{\rm{Ext}}^{1}(M_{\alpha},M_{\beta})=0$ for any
 $\alpha\in \mathcal{I}_{t},\beta\in \mathcal{I}_{u},t\leqslant u$.
 Hence for any representation $V$, it decomposes as $V=\oplus_{t=1}^{s}M_{(t)}$,
 where $M_{(t)}$ is the direct sum of all indecomposable summands of $V$ which
 are isomorphic to some $M_{\alpha}$ for $\alpha\in \mathcal{I}_{t}$.
 Reineke associates a sequence $\bm{i}$ to a directed partition
 $\mathcal{I}_*=(\mathcal{I}_1,\ldots,\mathcal{I}_s)$ as follows.
 For $t=1,\ldots,s$, he defines a sequence $\omega_t$ of vertices by writing the elements
 of the set $\{i\in Q_{0}: \alpha_{i}\neq0 \text{ for some }\alpha\in \mathcal{I}_{t} \}$
 in ascending order. Now $\bm{i}$ is the concatenation $\omega_1\ldots\omega_s$.
 Furthermore, a sequence $\bm{a}$ is defined as $\bm{a}=(\bm{a}_{1},\dots,\bm{a}_{s})$, where
 $\bm{a}_{t}=(\underline\dim{}_{i_{1}}M_{(t)},\dots,\underline\dim{}_{i_{b}}M_{(t)})$
 for $\omega_{t}=(i_{1},\dots,i_{b})$. Then for this pair $(\mono)$ Reineke proved that
 $\pi_{\mono}:X_{\mono}\rightarrow \overline{\mathcal{O}_{V}}$ is a resolution of
 singularities \cite[Theorem~2.2]{Rei03}.
 
 For cyclic quiver $\tilde{A}_{r-1}$, any indecomposable nilpotent representation is
 isomorphic to some $S_{i}[l]$ (a representation of total dimension $l$ with the cosocle
 at the vertex $i\in{\mathbb Z}/r{\mathbb Z}$). Let $(V,\varphi)\in E_{d}$ be a
 representation, then
 $0\subset \K\varphi\subset\K\varphi^{2}\subset\cdots\subset\K\varphi^{\nu}=V$
 gives a filtration of $V=\oplus_{i\in \mathbb{Z}/r\mathbb{Z}}V_{i}$. Let
 $m^{t}=\underline\dim\K\varphi^{t}/\K\varphi^{t-1}$, and $\mathcal{F}_{m}$ be the variety of all flags
 $$0=F_{0}\subset F_{1}\subset\cdots\subset F_{\nu}=V$$
 of ${\mathbb Z}/r{\mathbb Z}$-graded vector spaces such that
 $\underline\dim F_{t}/F_{t-1}=m^{t}$. Then the variety $X\subset E_{d}\times \mathcal{F}_{m}$
 formed by all the pairs $(\varphi,F)$ such that $\varphi(F_{t})\subset F_{t-1}$
 can be identified with $X_{\mono}$ for some $(\mono)$. Schiffmann proved that $\pi_{\mono}:X_{\mono}\rightarrow \overline{\mathcal{O}_{V}}$ is a resolution of singularities \cite[Proposition~1.1]{Sch04}. In particular, let $d=(n,\dots,n)$ and $\varphi$ be the indecomposable representation $S_{1}[rn]$, then we get the vector bundle $\mathcal{T}_{r}^{*}\mathcal{B}_{N}^{r}$. Hence Proposition \ref{birational} is a particular case of \cite[Proposition~1.1]{Sch04}.

 Combining the above arguments with Panyushev's positivity results,
 we obtain for cyclic and ADE quivers the following strengthening of~Theorem~\ref{51}:

 \begin{Thm}
   Let $Q$ be an $ADE$ or a cyclic quiver, and let $(\mono)$ be the sequences described in the
   above paragraphs. Then
$K_{\bm{\lambda},\bm{\mu}}^{\mono}(q_{Q_{1}})$ has positive coefficients for any $\bm{\lambda}$ dominant and $\bm{\mu}\gtrdot |Y_{\mono}|-|\mathfrak{n}_{P_{\mono}}|$.
\end{Thm}

\section*{Acknowledgement}
I got to know this question from Professor Michael Finkelberg when he visited AMSS and gave a course on affine Grassmannians for two weeks in January of 2017.
I am very grateful to him for stimulating discussions, and to Professor Wen-Wei Li who pointed out to me the relation of this paper to \cite{Pan10}.
Thanks also for Professor Nanhua Xi and Professor Sian Nie's helpful suggestions on the first draft of this paper.

{
\bibliography{hy}
\bibliographystyle{alpha}
}
\end{document}